\documentclass{amsart}
\usepackage{amsmath,amssymb}
\usepackage{color}
\usepackage{hyperref}
\usepackage{graphicx}

\usepackage{enumerate}
\numberwithin{equation}{section}

\newcommand{\lalpha}{\underline{\alpha}}
\newcommand{\ualpha}{\overline{\alpha}}
\newcommand{\ldelta}{\underline{\delta}}
\newcommand{\udelta}{\overline{\delta}}

\def\<{\langle}
\def\>{\rangle}

\newcommand{\mP}{\mathcal{P}}

\newcommand{\rad}{{\rm rad}}
\newcommand{\R}{{\mathbb R}}
\newcommand{\bS}{{\mathbb S}}
\newcommand{\N}{{\mathbb N}}
\newcommand{\Z}{{\mathbb Z}}
\renewcommand{\P}{{\mathbb P}}

\newcommand{\ind}[1]{\mathbf{1}_{#1}}

\newcommand{\sH}{\mathcal{H}}

\newcommand{\sF}{\mathcal{F}}
\newcommand{\sD}{\mathcal{D}}
\newcommand{\sA}{\mathcal{A}}
\newcommand{\sK}{\mathcal{K}}

\newcommand{\sq}{\square}
\newcommand{\bz}{\mathbf{z}}
\newcommand{\1}{\mathbf{1}}

\newcommand{\diam}{\operatorname{diam}}

\newcommand{\conv}{\operatorname{conv}}

\newcommand{\bd}{\partial}
\newcommand{\graphG}{G}

\newcommand{\sM}{\mathcal{M}}
\theoremstyle{plain}
   \newtheorem{thm}{Theorem}[section]
    \newtheorem{cor}[thm]{Corollary}
    \newtheorem{lem}[thm]{Lemma}
     \newtheorem{prop}[thm]{Proposition}
\theoremstyle{definition}
   \newtheorem{defn}[thm]{Definition}
   \newtheorem{rem}[thm]{Remark}
\newtheorem{ex}[thm]{Example}

\begin{document}

\title[Radial growth of ballistic aggregation]{On the radial growth of ballistic aggregation and other aggregation models}

\author{Tillmann Bosch}\curraddr{\it Tillmann Bosch, Langstr. 43, 68169 Mannheim Germany} \email{tillmannbosch3@gmail.com}
\author{Steffen Winter}\curraddr{\it Steffen Winter, Institute of Stochastics, Karlsruhe Institute of Technology, Englerstr. 2, D-76131 Karlsruhe, Germany}\email{steffen.winter@kit.edu}

\begin{abstract}
For a class of aggregation models on the integer lattice $\Z^d$, $d\geq 2$, in which clusters are formed by particles arriving one after the other and sticking irreversibly where they first hit the cluster, including the classical model of diffusion-limited aggregation (DLA), we study the growth of the clusters. We observe that a method of Kesten used to obtain an almost sure upper bound on the radial growth in the DLA model generalizes to a large class of such models. We use it in particular to prove such a bound for the so-called ballistic model, in which the arriving particles travel along straight lines. Our bound implies that the fractal dimension of ballistic aggregation clusters in $\Z^2$ is 2, which proves a long standing conjecture in the physics literature.
\end{abstract}

\thanks{Steffen Winter was supported by the DFG grant 433621248.}


 \date{\today}



\maketitle

\section{Introduction}\label{sec:intro}

Consider a process of cluster formation on the hypercubic lattice $\Z^d$, $d\geq 2$, which starts at time one with a single particle placed at the origin. Then at each time step $n=2,3,\ldots$, a particle arrives and is added to the cluster by placing it at a site neighbouring the existing cluster. The position of the new particle is chosen at random independently of all previous choices and according to some distribution which only depends on the existing cluster at that time. We will call such aggregation models \emph{incremental aggregation} in the sequel. Many such models are known and have been studied intensively in the physics literature. They differ only by the choice of the distributions governing the selection of the locations for the arriving particles.

The most popular among these models is certainly \emph{diffusion-limited aggregation (DLA)}, suggested by Witten and Sander \cite{WittenSander81} in 1981 as a model for metal-particle aggregation. Particles arrive from `far away' and perform symmetric random walks until they hit the cluster for the first time. The distribution specifying the resulting cluster formation mechanism is known as the harmonic measure. The formed clusters have a dendrite-like structure and show a fractal behaviour, see Figure~\ref{fig1}.
Another even older model is the \emph{Eden growth model} due to Murray Eden \cite{Eden1961}, in which the position of the next particle is chosen uniformly out of all sites neighboring the current cluster. The clusters formed in this model tend to be very dense and appear ball like with a rough boundary.
Another variation is \emph{diffusion-limited annihilation} due to Meakin and Deutch \cite{Meakin-Deutch86}, nowadays better known as \emph{internal DLA}. Here the particles start at the origin, i.e.\ within the cluster, perform symmetric random walks and are attached where they exit the cluster. As in the Eden model, the formed clusters tend to be ball-like with a rough boundary.

Another popular model in the physics literature is the \emph{Vold-Sutherland model} or \emph{ballistic aggregation}, in which arriving particles travel along straight lines (ballistic trajectories) and stick where they hit the cluster. It can be traced back to the work of the colloid chemists Marjorie J.\ Vold and David N.\ Sutherland, see e.g.~\cite{Meakin83,Vicsek89}. It has been suggested for situations in which molecules move in a low density vapour such that thermal interactions can be disregarded. The clusters formed by this model are denser than the ones formed by DLA, but have a similar dendritic structure, see Figure~\ref{fig1}.

\begin{figure} \label{fig1}
\includegraphics[width=0.49\textwidth]{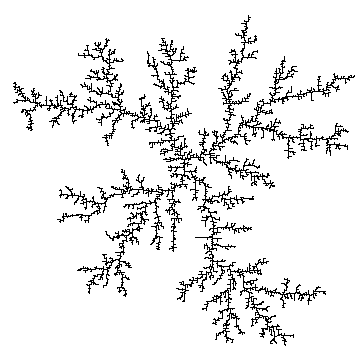}
 \includegraphics[width=0.49\textwidth]{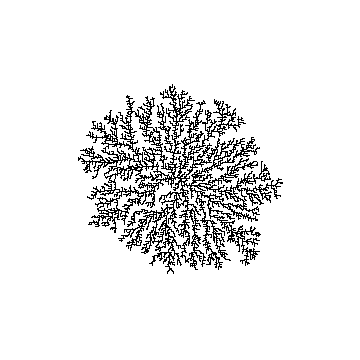}
 \caption{Typical realization of an aggregation cluster of DLA (left) and the ballistic aggregation model (right) of size $n=10000$.}
\end{figure}
All these models and many variants have been studied intensively in physics and have found numerous applications, we refer to the surveys and books \cite{Vicsek89,Stanley-Ostrowski86,Jullien-Botet87,Kolb2001,barabasi-stanley1995} for further details and references. Most results have been obtained either from simulations or from intriguing but non-rigorous theoretical considerations. Some of these models have also inspired mathematical research and a number of results have been obtained.  However, the progress up to now is far behind what physicists have discovered.

One of the most basic questions (to which we will mainly restrict our attention in this paper) is the speed of growth of the radius (or diameter) of the cluster as the number of particles tends to $\infty$. Regarding DLA, some rigorous almost sure bounds are due to Kesten \cite{Kesten87,Kesten90} (with some improvements for dimension 3 in \cite{lawler91} and recently in \cite{benjamini2017upper}), which we will recall in Section \ref{sec:growth} below. They imply in particular that DLA clusters are not full-dimensional (in a discrete sense) and thus exhibit some fractal behaviour. For internal DLA (and the Eden model), the growth rate is $n^{1/d}$ in $\Z^d$, as the number of particles tends to $\infty$ and thus the clusters are full-dimensional. For internal DLA even some shape theorems have been established showing that asymptotically the clusters are balls \cite{LBG1992-IDLA}. For the Eden model the existence of a limit shape is also known, due to some relation of this model to first passage percolation, cf.~\cite{Auffinger-etal2017} and the references therein. The limit shape is conjectured to deviate from a Euclidean ball. Up to now this has only been established in very high dimensions. We refer to the survey article \cite{Sava-Huss21} for an overview of further results regarding DLA und internal DLA. In particular, in this article, attention is also given to results obtained on graphs other than the hypercubic lattices $\Z^d$ that we restrict ourselves to here. (It is obvious that the models may be transferred to any infinite connected graph.)
In contrast to the situation for the other mentioned models, it seems that the ballistic model has not been investigated in the maths literature so far. To the best of our knowledge, the only paper, where it is mentioned (but not analyzed very far) is the article \cite{Molchanov07}, in which also a number of new aggregation models are proposed and simulated.

In this article we make a first attempt to investigate rigorously the growth behaviour of the ballistic model. To this end, we start from a more general viewpoint and set up a natural class of models (incremental aggregation), which includes the models mentioned above (and many more). We revisit the argument of Kesten that he used in \cite{Kesten87} to obtain an upper bound on the growth rates of DLA in $\Z^2$ and observe that the same argument may be applied to any incremental aggregation model. In a nutshell, the `method'  reduces the problem of finding a bound for the radial growth to the much easier task to establish an upper bound on the local mass of the distributions determining the model. We do not claim much originality here, as once the class of incremental aggregation models is set up it is fairly easy to see that Kesten's argument works in general.
Then we apply this new method to the ballistic model. First we give a rigorous definition of the ballistic model using some notions from geometric probability. Then we state some bound on the local mass for the ballistic measures (i.e., the distributions defining the ballistic model) and finally we apply Kesten's method to obtain the desired growth rates. It turns out that the method has power in particular in dimension $d=2$. For the ballistic model in $\Z^2$ we establish that the growth exponent is $1/2$ (i.e., the fractal dimension is 2), which is the value conjectured in the physics literature. For $d\geq 3$, our results imply that the fractal dimension is bounded from below by 2 (while the conjectured value is $d$).

In order to illustrate our results, we have also simulated aggregation clusters of the ballistic model (as we define it here) and DLA for comparison. The Python code is freely available, cf.~\cite{Bosch-GitHub}, together more pictures of large simulated clusters.

The paper is organized as follows. In Section~\ref{sec:main}, we introduce incremental aggregation and show that this class of models includes DLA and the Eden model. In Section~\ref{sec:growth} we provide some definitions regarding radial growth and state some trivial facts about growth exponents and fractal dimensions of our models. We also recall the results of Kesten on the radial growth of DLA. In Section~\ref{sec:Kesten}, we state Kesten's method (Theorem~\ref{thm:Kestens-method}) for general incremental aggregation and demonstrate how it can be applied to recover Kesten's growth bounds for DLA and the known bounds for the Eden model.
In Section~\ref{sec:ballist}, we define the ballistic model and apply Kesten's method to obtain the mentioned results on the growth of these models. Finally, in Section~\ref{sec:proofs} we provide a proof of our main tool, Kesten's method.

\section{Incremental aggregation}\label{sec:main}

We set up a framework in which all the models can be treated in a unified way.
 Let $\mP^d_f$ denote the family of finite subsets of $\Z^d$, i.e.
  \begin{align*}
	\mP^d_f:= \{A\subset \Z^d\ |\ \text{A is finite}\},
\end{align*}
which we equip with the discrete $\sigma$-algebra.
 We consider the {\em nearest neighbor graph} on $\Z^d$, i.e., the graph $(\Z^d,E)$ with vertex set $\Z^d$ and edge set $E:=\{\{x,y\}\subset \Z^d: \|x-y\|=1\}$. (Here and throughout $\|\cdot\|$ denotes the Euclidean norm.)
 A set $A\in\mP^d_f$ is called {\em connected}, if the subgraph of $(\Z^d, E)$ generated by $A$ is. For $A\in\mP^d_f$, the {\em (outer) boundary} of $A$ is the set
  \begin{align*}
    \bd A:=\{y\in\Z^d\setminus A: \exists x\in A \text{ such that } \{x,y\}\in E\}.
  \end{align*}
   Throughout let $(\Omega,\sA,\P)$ be some suitable probability space. For $A\in\mP^d_f\setminus\{\emptyset\}$, a {\em random point in $A$} is a measurable mapping $y_A:\Omega\to \Z^d$ with $\P(y_A\in A)=1$.
Denote by $\sD(A)$ the family of all probability measures on $A$, i.e., of all possible distributions of a random point in $A$.
Moreover, a {\em random finite set} is a measurable mapping $F:\Omega\to\mP^d_f$.

\begin{defn} \label{incrementalaggregation}
	Let $\sM:=(\mu_A)_{A\in \mP^d_f}$ be a family of distributions s.t.\ $\mu_A\in\sD(A)$ for each $A\in \mP^d_f$.
A sequence $(F_n)_{n\in\N}$ of random finite sets $F_n\subset\Z^d$ is called {\em incremental aggregation (with distribution family $\sM$)}, if it satisfies the following conditions:
\begin{enumerate}
 \item[(i)] $F_1:=\{y_1\}$, where $y_1:=0\in\Z^d$;
 \item[(ii)] for any $n\in\N$, $F_{n+1}:=F_n\cup\{y_{n+1}\}$, where $y_{n+1}$ is a random point in $\Z^d$ whose conditional distribution given $F_n$ is $\mu_{\partial F_n}$, i.e.,
     \begin{align*}
		\P\left(y_{n+1} = y\ | F_n=A\right) := \mu_{\partial A}(y) \qquad \text{ for any } A\in\mP_f^d \text{ and } y\in \Z^d.
	\end{align*}
\end{enumerate}
\end{defn}

$F_n$ is called {\em cluster} or {\em aggregate at time $n$}, and $F_\infty := \bigcup_{n\in\N} F_n$ the {\em infinite cluster}.
Observe that
 $$
 0\in F_1\subset F_2 \subset F_3 \subset \ldots\subset F_\infty\subseteq \Z^d.$$
  Moreover, for any $n\in\N$, almost surely $F_n$ is connected and $\# F_n=n$.
It is easy to see that $(F_n)_n$ is a Markov chain, in particular, $F_{n+1}$ depends on $F_n$, but not on the order, in which the points have been added to $F_n$.
Different aggregation models arise now by choosing different families $\sM$ of distributions. We start with the  simple example of the Eden growth model.

\begin{ex}
  [Eden growth model] \label{ex:eden}
  For each $A\in\mP^d_f$, let $\mu_A$ be the uniform distribution on $A$. Then incremental aggregation with distribution family $\sM=(\mu_A)_{A\in\mP^d_f}$ is known as \emph{Eden growth model}. Here, given $F_n=A\in\mP^d_f$, each point $y$ in the outer boundary $\bd A$ of $A$ has exactly the same chance (namely $1/{\# \bd A}$) of being the next point $y_{n+1}$ added to the cluster. From simulations one can observe that large clusters look ball-like with few holes and a rough boundary. 
\end{ex}

Our second example of incremental aggregation is diffusion limited aggregation. It has been introduced by Witten and Sander \cite{WittenSander81} as a model for metal-particle aggregation and became soon very popular in physics as a simple model for many different phenomena, where particles aggregate irreversibly and diffusion (thermal motion) is
the means of particle transport.
\begin{ex}
[Diffusion limited aggregation (DLA)] \label{ex:dla}
   Let $(S^x_n)_{n\in\N_0}$ be a symmetric random walk on $\Z^d$ started at $x\in\Z^d$, i.e., $\P(S^x_0=x)=1$ and 
      for each $n\in\N$ and each $z\in\Z^d$
$$
\P(S^x_n=y|S^x_{n-1}=z)=\frac 1{2d}, \text{ for all neighbors } y \text{ of } z.
$$
  For $A\subset\Z^d$, let $T^x_A:=\min\{n\in\N: S^x_n\in A\}$ be the {\em hitting time of $A$.}
  It is well known that in $\Z^2$, a random walk started at $x$ is {\em recurrent}, i.e.\ $\P(T^x_{\{x\}}<\infty)=1$, and in $\Z^d$, $d\geq 3$ it is {\em transient}, i.e.\ $\P(T^x_{\{x\}}<\infty)<1$.
 For $A\in\mP^d_f$ and any point $x\in \Z^d\setminus A$, define a distribution on $A$ by
  $$
  H^x_{A}(y):=\P(S^x_{T^x_A}=y| T^x_A<\infty), \qquad y\in A.
  $$
It is a nontrivial but well established fact that by setting
  $$
  h_A(y):=\lim_{\|x\|\to\infty} H^x_{A}(y), \qquad y\in A
  $$
  a probability measure $h_A$ is defined on $A$. It is called the {\em harmonic measure} of $A$.
{\em DLA} is now defined to be incremental aggregation with distribution family $(h_A)_{A\in\mP^d_f}$.
So in the DLA model particles perform random walks started at infinity and stick where they first hit the existing cluster. This produces sparse dendrite-like structures as simulations show.
\end{ex}

Below in Section~\ref{sec:ballist} we will discuss another class of models called ballistic aggregation. In these models the particles travel along straight lines instead of random walk paths. In order to introduce them properly, we will need some concepts from geometric probability.

\section{Growth of arms}\label{sec:growth}

Our aim is to study the speed of growth of aggregation clusters in incremental aggregation models. For any finite set $A\subset\Z^d$ with $0\in A$, let
  \begin{align}
     \label{eq:def-rad}
  \rad(A):=\max_{x\in A} \|x\|
  \end{align}
  denote its {\em radius}, where $\|\cdot\|$ denotes the Euclidean norm. 
Observe that
$$
\rad(A)=\inf\{r>0: A\subset B(0,r)\}, 
$$
justifying the terminology. Denote by $\diam(A)$ the \emph{diameter} of $A$ and note that
   \begin{align}
     \label{eq:rad-diam} \frac 12 \diam(A)\leq \rad(A)\leq \diam(A).
   \end{align}
   Instead of the radius one could similarly work with the diameter in the sequel, which might be more natural in any more general setting. Since we are only interested in incremental aggregation, which is started at $0$, we stick to the radius for historic reasons.
    The following simple observation bounding the radius of a finite set $A\subset\Z^d$ in terms of its cardinality $\# A$ turns out to very useful.
     \begin{lem} \label{lem:rad-bounds}
     Let $A\subset \Z^d$ be a finite, connected set with $0\in A$. Then
   \begin{align}
     \label{eq:rad-bounds} \frac 12 (\# A)^{1/d}\leq \rad(A)\leq \# A.
   \end{align}
   \end{lem}
   \begin{proof}
      Let $r:=\rad(A)$. Then all points of $A\subset\Z^d$ are contained in the ball $B(0,r)\subset [-r,r]^d$. Hence $\# A\leq (2r)^d$, implying the first inequality in \eqref{eq:rad-bounds}.  For the second one note that since $0\in A$  and $A$ is connected, $A$ must contain a path $P$ from $0$ to some point $y\in A$ with $\|y\|=r$. Such a path contains at least $r$ points and thus $\# A\geq\# P\geq r$, showing the second inequality in \eqref{eq:rad-bounds}.
   \end{proof}

Note that the constant $\frac 12$ in \eqref{eq:rad-bounds} could be improved but we do not need it in the sequel. Note also that for the first inequality the assumption that $A$ is connected and contains $0$ is not needed.

   Let $(A_n)_{n\in\N}$ be an increasing sequence of finite subsets of $\Z^d$.
We are looking for a suitable exponent $\alpha$ (the {\em `growth rate'} of the sequence) and a suitable constant $c>0$ such that
  $$
  \rad(A_n)\sim c\cdot (\# A_n)^\alpha, \quad \text{ as } n\to\infty.
      $$
      (Here $a_n\sim b_n$ means that the quotient $a_n/b_n$ converges to 1 as $n\to\infty$.)
Since such exponent $\alpha$ might not exist in general, we define the {\em lower and upper growth rate} of the sequence $(A_n)$ by
  \begin{align*}
    \lalpha_f:=\lalpha_f((A_n)):=\liminf_{n\to\infty}\frac{\log(\rad(A_n))}{\log(\# A_n)} \quad\text{ and } \quad \ualpha_f:=\limsup_{n\to\infty}\frac{\log(\rad(A_n))}{\log(\# A_n)}.
  \end{align*}
  Alternatively, the {\em lower and upper fractal dimension} are defined by
  \begin{align*}
    \ldelta_f:=\liminf_{n\to\infty}\frac{\log(\# A_n)}{\log(\rad(A_n))}
    \quad \text{ and }\quad
     \udelta_f:=\limsup_{n\to\infty}\frac{\log(\# A_n)}{\log(\rad(A_n))}.
  \end{align*}
  It is immediately clear from the definitions that fractal dimension and growth rate are directly connected. In general one has $\ldelta_f=1/\ualpha_f$ and $\udelta_f=1/\lalpha_f$.
  This means in particular that an upper bound on the growth rate implies a lower bound on the fractal dimension and vice versa.

  $\# A_n$ may be interpreted as the \emph{volume} of the cluster $ A_n$. (We may identify $\# A_n$ with the union $\square A_n:=\bigcup_{y\in A_n} C_y$, where $C_y:=y+[-\frac 12, \frac 12]^d$ is the unit cube centered at $y$. Then the volume of $\sq A_n$ is exactly $\# A_n$.) Moreover, $\rad(A_n)$ may be interpreted as the diameter, cf.\ \eqref{eq:rad-diam}, which justifies calling the exponents $\ldelta_f$ and $\udelta_f$ fractal dimensions.

  The trivial bounds on the radius stated in Lemma~\ref{lem:rad-bounds} imply immediately some general bounds on the fractal dimensions and growth rates.
  \begin{prop}
    \label{prop:dim-triv-bounds}
    Let $(A_n)_{n\in\N}$ be an increasing sequence of finite, connected subsets of $\Z^d$ with  $0\in A_1$. Then
  $$
  1\leq \ldelta_f\leq \udelta_f\leq d \qquad \text{ and hence }\qquad 1/d\leq \lalpha_f\leq \ualpha_f\leq 1.
  $$
  \end{prop}
\begin{proof}
   The lower bound 1 for $\ldelta_f$ follows by applying the right hand side inequality in \eqref{eq:rad-bounds} to $A_n$ and taking the $\liminf$. The upper bound for $\udelta_f$ follows similarly by applying the left hand side inequality in \eqref{eq:rad-bounds} to $A_n$ and taking the $\limsup$.
\end{proof}

In incremental aggregation models, the sequence $(F_n)_{n\in\N}$ consists of random sets and therefore, for each $n\in\N$, $\rad(F_n)$ is a nonnegative random variable (while $\# F_n=n$ is constant almost surely). Hence the fractal dimensions and growth rates are random variables, too. Observe that they satisfy almost surely the bounds stated in Proposition~\ref{prop:dim-triv-bounds}.

Let now $(F_n)_{n\in\N}$ be DLA in $\Z^d$, $d\geq 2$ as defined in Example~\ref{ex:dla}. The following celebrated result due to Kesten provides an almost sure upper bound on the radii of DLA clusters. For $d=3$, a log-term appears which was improved slightly by Lawler \cite{lawler91} and later again by Benjamini and Yadin \cite{benjamini2017upper}.

  \begin{thm}[cf.~\cite{Kesten87,Kesten90,benjamini2017upper}] \label{thm:DLA}
     For DLA $(F_n)_n$ in $\Z^d$ there exists a constant $c>0$ (depending only on $d$) such that almost surely for $n$ sufficiently large
     \begin{align*}
       \rad(F_n)\leq\begin{cases}
          c\, n^{2/3},& \text{if } d=2,\\
          c\, n^{1/2}(\log n)^{1/2}, &\text{if } d=3,\\
          c\, n^{2/(d+1)}, & \text{if } d\geq 4.
       \end{cases}
     \end{align*}
  \end{thm}
  Theorem~\ref{thm:DLA} implies the following lower bounds on the fractal dimension.
   \begin{cor} \label{cor:dla-dim}
       For DLA $(F_n)_{n\in\N}$ in $\Z^d$, $d\geq 2$, one has almost surely
        \begin{align*}
       \ldelta_f\geq
                 \frac{d+1}2 \quad \text{ and } \quad \ualpha_f\leq \frac 2{d+1}.   
       \end{align*}
   \end{cor}
   \begin{proof}
      This follows immediately by inserting the estimate in Theorem~\ref{thm:DLA} into the definition of the lower fractal dimension (and the upper growth exponent). For $d=2$, for instance, this yields almost surely
      \begin{align*}
     \ldelta_f=\liminf_{n\to\infty}\frac{\log n}{\log(\rad(F_n))}\geq \lim_{n\to\infty}\frac{ \log n}{\log(c n^{2/3})}
     = \frac 32. \qquad \qedhere 
  \end{align*}
   \end{proof}

   Simulations of the model suggest that the above bounds are not sharp. In $\Z^2$ they suggest that the fractal dimension exists and equals $\delta_f=\frac 53\approx 1.66$, which is strictly larger than the rigorous bound stated above. For general $d\geq 2$, it is conjectured, cf.~e.g.~\cite{Vicsek89}, 
   that for DLA in $\Z^d$ the dimension is given by
   $$
   \delta_f=\frac{d^2+1}{d+1}.
   $$

\section{DLA and Kesten's method} \label{sec:Kesten}

In his paper \cite{Kesten87}, Kesten used a certain strategy of proof for his bounds for the DLA model in $\Z^2$, which is also described in Lawler's book, see \cite[§ 2.6]{lawler91}, and generalized to arbitrary graphs with bounded degree in \cite{benjamini2017upper}. It turns out that this method can be adapted to provide some bounds for any incremental aggregation model. The next theorem below describes this general method. Its proof will be discussed later in Section~\ref{sec:proofs}.
Let $\mP^d_{\bd}$ be the family of outer boundaries of aggregation clusters in $\Z^d$, that is, let
\begin{align*}
  \mP^d_{\bd}:=\{B\in\mP^d_f: \exists A\in\mP^d_f \text{ such that } A \text{ is connected, } 0\in A \text{ and } \bd A =B\}.
\end{align*}
Recall that in any incremental aggregation model, at any time step $n$ the next point $y_{n+1}$ to be attached is chosen from the outer boundary $\bd F_n$ of the current cluster $F_n$. Therefore, it is in fact sufficient to specify the distributions $\mu_B$ for all $B\in\mP^d_{\bd}$ in order to determine an incremental aggregation model completely. As a consequence, it is sufficient to impose conditions on the distribution family only for the distributions of outer boundary sets as in the statement below. Other sets are not relevant for the model. 
\begin{thm}[Kesten's method] \label{thm:Kestens-method} 
    Let $\sM=(\mu_A)_{A\in\mP^d_f}$ be some family of distributions such that $\mu_A\in \sD_A$ for each $A\in\mP^d_f$. Suppose there exist some positive constants $q$ and $C$ 
    such that for all $r>1$, any $B\in\mP^d_{\bd}$ with radius at least $r$ 
    and any $z\in B$,
\begin{align} \label{eq:q-estimate}
  \mu_B(z)\leq C\, r^{-q}.
\end{align}
Then there is a constant $c$, such that incremental aggregation $\sF=(F_n)_{n\in\N}$ with distribution family $\sM$ satisfies almost surely
\begin{align*}
   \rad(F_n)\leq c\, n^{1/(q+1)}
\end{align*}
for $n$ sufficiently large.
\end{thm}

The method is rather crude. Roughly it says that if for all relevant distributions in the family $\sM$  the mass concentrated on single vertices is not too large compared to the radius of the corresponding cluster (implying that the probability mass must be spread out over a significant number of vertices), then some bound on the radial growth follows. This is plausible. The more spread out distributions are, the more options there are for the next particle to be placed. As only few of the possible locations are extremal in the sense that they lead to radial growth of the cluster, this limits the speed of growth.

In \cite{benjamini2017upper}, the assumption \eqref{eq:q-estimate} is called a $\varphi$-\emph{radius Beurling estimate}  in the context of the DLA model (for the function $\varphi(r)=c r^{-q}$). Our proof of Theorem~\ref{thm:Kestens-method} in Section~\ref{sec:proofs} below follows essentially the proof of Lawler~\cite[§ 2.6]{lawler91} but one might also want to compare it with \cite[Lemma~2.4 and Theorem~2.6]{benjamini2017upper}. The important observation is that the relevant arguments in these proofs do not only apply to DLA but to any incremental aggregation.

To illustrate Kesten's method, we briefly describe how it is applied to obtain the known bounds in $\Z^2$ for DLA and the Eden model. In the next section, we will use it to study ballistic aggregation.
For the DLA model in $\Z^2$, one can use the following well known estimate for harmonic measures.

\begin{prop}
  \label{prop:harmonic} (cf.\ e.g.~\cite[Proposition~2.5.2]{lawler91})
  For any $r>1$, any connected set $A\subset\mP^d_f$ with $0\in A$ and radius at least $r$, and any $z\in A$, one has
  \begin{align*}
       h_A(z)\leq\begin{cases}
          C\, r^{-1/2},& \text{if } d=2,\\
          C\, (\log r)^{1/2}r^{-1}, &\text{if } d=3,\\
          C\, r^{-1}, & \text{if } d\geq 4.
       \end{cases}
     \end{align*}
\end{prop}
Observe that outer boundaries need not be connected such that the above estimates are not directly applicable. However, since any random walk started outside a set $A$ will first hit $\bd A$ before hitting $A$, we have
\begin{align*}
  h_{\bd A}=h_{A\cup\bd A}   \quad \text{ for any } A\in\mP^d_f.
\end{align*}
Hence the above estimates hold also for any boundary set $B\in\mP^d_\bd$.
Applying Theorem~\ref{thm:Kestens-method} to DLA in $\Z^2$, for which, by Proposition~\ref{prop:harmonic}, the hypo\-thesis is satisfied for $q=1/2$, we conclude the existence of a constant $c>0$ such that almost surely
$
   \rad(F_n)\leq c\, n^{3/2}
$
for $n$ sufficiently large. This is the bound stated in Theorem~\ref{thm:DLA} above.
\begin{rem}
  For DLA in $\Z^3$, the bound for the harmonic measure stated in Proposition~\ref{prop:harmonic} implies that the hypo\-thesis of Theorem~\ref{thm:Kestens-method} is satisfied for any $q<1$. Thus, this theorem yields for any $p>\frac 12$ the existence of a constant $c_p$ such that a.s.\ the bound $
   \rad(F_n)\leq c_p\, n^{p}
$ holds for $n$ sufficiently large. This is enough to conclude (by letting $p\to 1/2$) that a.s.\ $\ldelta_f\geq 2$, as stated in Corollary~\ref{cor:dla-dim}, but it is not enough to get the logarithmic correction for the bound on the radius stated in Theorem~\ref{thm:DLA} for the case $d=3$. For this a slight refinement of Kesten's method is necessary, see e.g.~\cite{benjamini2017upper}.

For $d\geq 4$, Theorem~\ref{thm:Kestens-method} yields the  existence of some $c>0$ such that a.s.\ the bound $
   \rad(F_n)\leq c\, n^{1/2}
$ holds for $n$ sufficiently large. This is not as good as the bound stated in Theorem~\ref{thm:DLA} and shows the rather poor performance of the method in higher dimensions. It is known that the exponent $q=1$ in Proposition~\ref{prop:harmonic} is optimal. There are sets $A$ (e.g.\ (discrete) line segments) for which the harmonic measure has atoms of this order, cf.\ e.g.~\cite[§2.4]{lawler91} for details. Hence Theorem~\ref{thm:Kestens-method} cannot provide a better bound for the growth rate in this case.
\end{rem}

For the Eden growth model in $\Z^d$ (as defined in Example~\ref{ex:eden} above), we have the following estimate for the defining family of distributions $(\mu_{A})$:

\begin{lem} \label{lem:eden}
  For any $r>1$, any boundary set $B\in\mP^d_{\bd}$ with $\rad(B)\geq r$
  and any $z\in B$,
\begin{align*}
  \mu_B(z)\leq  \sqrt{d}(2(d-1))^{-1} r^{-1}.
\end{align*}
\end{lem}
\begin{proof}
Denote by $e_1,\ldots,e_d$ the coordinate directions.
   For $B\in\mP^d_{\bd}$ there is a finite connected set $A\in\mP^d_f$ with $0\in A$ and $\rad(A)\geq r-1$, such that $\bd A=B$.
   Because of its radius $A$ must contain a point $y$ with $\|y\|\geq r-1$. Hence there is a coordinate direction $e_i$ such that the $i$-th coordinate of $y$ satisfies $|y_i|\geq (r-1)/\sqrt{d}$. Without loss of generality we can assume $i=1$ and $y_1>0$. Since $y_1$ is an integer, we have $y_1\geq \lfloor (r-1)/\sqrt{d} \rfloor+1$ (where $\lfloor x\rfloor$ denotes the integer part of $x>0$) such that $y_1+1\geq (r-1)/\sqrt{d}+1\geq r/\sqrt{d}$.

   Since $A$ is connected, also its projection onto the $e_1$-axis must be connected. Hence, for each $k=0,\ldots, y_1$ there is a vertex $z_k$ in $A$ whose first coordinate is $k$. If one moves from $z_k$ in direction $\pm e_i$, $i=2,\ldots, d$ to the first vertex outside $A$, then this new vertex will be in $\bd A=B$. It is clear that all the $2(d-1)(y_1+1)$ vertices of $B$ reached in this way are distinct, which implies that $$
   \# B\geq 2(d-1)(y_1+1)\geq \frac {2(d-1)}{\sqrt{d}} r.$$  Hence $\mu_B(z)=(\# A)^{-1}\leq \sqrt{d}(2(d-1))^{-1} r^{-1}$ as asserted.
\end{proof}

Applying now Kesten's method to the Eden growth model $(F_n)_n$ in $\Z^d$ using the estimate provided in Lemma~\ref{lem:eden} (for the exponent $q=1$), we infer that there
is a constant $c>0$, such that almost surely
\begin{align*}
   \rad(F_n)\leq c\, n^{1/2}
\end{align*}
for $n$ sufficiently large. This implies $\ldelta_f\geq 2$ for the lower fractal growth dimension of the Eden model in $\Z^d$. For $d=2$, by Proposition~\ref{prop:dim-triv-bounds}, we therefore recover
$$
\delta_f=2.
$$
It is also well known that $\delta_f=d$ for the Eden model in $\Z^d$, $d\geq 3$, but Kesten's method as stated in Theorem~\ref{thm:Kestens-method} is not capable of providing such a result, as this would require an estimate of the form \eqref{eq:q-estimate} with $q=d-1$, which is simply not true. Indeed, for any $r\in\N$ there is a cluster $A$ such that for $B=\bd A$, $\rad(B)=r$ and $\mu_B(z)\geq (2(d-1))^{-1} r^{-1}$. (Take e.g.\ $A:=\{(k,0,\ldots,0)\in\Z^d: k\in\{0,1,\ldots, r-1\}$. Then $\# B= 2+2(d-1)r\leq 2(d-1)r$ and so $\mu_B(z)\geq (2(d-1))^{-1} r^{-1}$.) This precludes an estimate of the form \eqref{eq:q-estimate} to hold for any $q>1$.

The hypothesis in Kesten's method is rather restrictive. Requiring some deterministic estimate to be satisfied for all outer boundaries and at all locations $z$ leads to an exponent $q$ that is often too small in order to provide a good bound for the radial growth. Although the methods is able to provide lower bounds on radial growth in any dimension $d$, it seems that the method is strong only in $\Z^2$.

 It seems plausible, that it should be enough to have a bound on $\mu_{\bd A}(z)$ available for outer boundaries of 'typical' aggregation clusters $A$ (or for 'most of them'). However, this will require further investigation and is not covered by Theorem~\ref{thm:Kestens-method} above. Another approach (which leads in fact to the growth bounds for DLA stated above) are estimates in terms of the volume of the clusters rather than its radius, which will be a topic of further investigation.

\section{Ballistic aggregation}\label{sec:ballist}

We now introduce the ballistic model rigorously and discuss its growth properties.
Ballistic aggregation in $\Z^d$ will be defined as incremental aggregation $(F_n)_{n\in\N}$ with a suitable distribution family $\sM=(b_A)_{A\in\mP^d_f}$. In order to determine the model all we have to do is to choose the family $\sM$, i.e., to fix a distribution $b_A$ on each finite set $A\in\mP^d_f$. Roughly, for each $z\in A$, $b_A(z)$ will be determined by the probability that a `directed random line through $A$ hits $z$ first'. In order to define this properly, we recall some ideas from stochastic geometry, in particular the concept of an \emph{isotropic random line}. In what follows, we will view $\Z^d$ also as a subset of $\R^d$ embedded in the natural way.

    Let $A(d,1)$ be the space of lines in $\R^d$ (i.e., the affine Grassmannian of $1$-flats). We equip it with the usual hit-or-miss topology and the associated Borel $\sigma$-algebra $\sA(d,1)$, see e.g.~\cite[Chapter 13]{SchneiderWeil} for details. For any compact set $K\subset\R^d$ let
    $$
        [K]:=\{L\in A(d,1): L\cap K\neq \emptyset\}.
        $$
  Then  $\sA(d,1)=\sigma(\{[K]: K\in\sK^d\})$, where $\sK^d$ denotes the family of all \emph{convex bodies}, that is, compact, convex sets in $\R^d$.

     It is well known, cf.~ e.g.\ \cite{SchneiderWeil}, that there is a unique Euclidean motion-invariant Radon measure $\mu_1$ on $A(d,1)$ such that
        $$
        \mu_1([B_d])=\kappa_{d-1}.
        $$
        Here $B_n$ is the unit ball in $\R^n$ and $\kappa_n:=\lambda_n(B_n)$ denotes its volume.

     Observe that, by the Crofton formula, cf.~e.g.~\cite[Theorem~5.1.1]{SchneiderWeil}, for any convex body $K\in\sK^d$,
     \begin{align}
       \label{eq:crofton}
       \mu_1([K])&=\int_{A(d,1)} \ind{} \{K\cap L\neq\emptyset\} \mu_1 (dL)\\
       &=\int_{A(d,1)} V_0(K\cap L) \mu_1 (dL)=\alpha_d\, V_{d-1}(K),\notag
     \end{align}
     where the constant is given by $\alpha_d:=\frac{2 (d-1)!\kappa_{d-1}}{d! \kappa_d}=\frac{2 \kappa_{d-1}}{d \kappa_d}$ and $V_{j}(K)$ is the intrinsic volume of $K$ of degree $j$. If $K$ has nonempty interior, then $V_{d-1}(K)$ is half the surface area of $K$. This means, the measure of lines hitting a given convex body $K$ is up to some universal constant given by the surface area of $K$.
     Following \cite[Definition 8.4.2]{SchneiderWeil}, for any convex body {{$K\in\sK^d$ with $V_{d-1}(K)>0$}},
      an {\em isotropic random line through $K$} is defined to be a random variable $L:\Omega\to A(d,1)$ with distribution given by
        $$
        \P(L\in \sA):=\P^K(\sA):=\frac{\mu_1(\sA\cap[K])}{\mu_1([K])},\qquad \sA\in\sA(d,1).
        $$
        This definition extends without any problem to arbitrary compact sets $K\subset\R^d$, provided that $\mu_1([K])>0$. However, the convenient interpretation \eqref{eq:crofton} of $\mu_1$ in terms of $V_{d-1}$ (or the surface area) is no longer valid if $K$ is not convex. Observe that for any compact $K\subset\R^d$, one has $\mu_1([K])\leq \mu_1\left(\left[\conv(K)\right]\right)<\infty$, where $\conv(K)$ denotes the convex hull of $K$. Thus, for any compact set $K\subset\R^d$ with $\mu_1([K])>0$, the distribution $\P^K$ is well defined and so is an \emph{isotropic random line through $K$}.
      In $\R^2$ one has the following additional property, which turns out to simplify the analysis significantly: for any connected, compact set $C\subset\R^2$,
     $$
     [C]=[\conv(C)].
     $$
     Therefore,
    $
    \P^C=\P^{\conv(C)}  
    $
    (provided $V_{1}(\conv(C))>0$, i.e., $C$ not a singleton). Thus in dimension 2 one can work with convex hulls instead of the original sets. Unfortunately, this is not possible in any higher dimension.

    For any finite set $A\in\mP^d_f$ denote by $\square A$ the union of grid boxes centered in $A$, that is, let
        $$
        \square A:=\bigcup_{z\in A} C_z \qquad\text{ where } C_z=\left[-\frac 12, \frac 12\right]^d+z.
        $$
        Then, for any $d\geq 2$ and $A\in\mP^d_f\setminus\{\emptyset\}$, the distribution $\P^{\sq A}$ of an isotropic random line through $\sq A$ is well defined.
        The idea for the definition of the ballistic measure $b_A$ is now as follows: We identify $A$ with $\sq A$ and generate an isotropic random line $L$ through $\sq A$. We choose randomly a direction on $L$. Then for any $z\in A$, we let $\mu_A(z)$ be the probability that, when traveling along $L$ in the chosen direction, the square $C_z$ is the first square in $\sq A$ visited by $L$. More formally, we fix for each $L\in A(d,1)$ a direction $v_L\in\mathbb{S}^{d-1}$. 
        (We choose $v_L$ such that the mapping $L\mapsto v_L$ is measurable.)
   For $A\in\mP^d_f$ and $L\in A(d,1)$, let
    $$
    L^A:=\{z\in A: C_z\cap L\neq\emptyset\},
    $$
    that is, $L^A$ is the set of those points in $A$ whose associated boxes are intersected by $L$. Traversing $L$ in direction $v_L$ induces an order in $L^A$. (For certain lines the order in which the boxes are visited is not well defined, namely if $L$ hits a box first at an intersection point of two or several boxes not visited before. In such a case the order can be made unique by fixing some order of the boxes $C_z$ in $z\in\Z^d$, such as the one induced by the lexicographic order in $\Z^d$. However, for $\mu_1$-almost all lines $L$ the order in $L^A$ is well defined without this.) Denote by $\min(L^A)$ and $\max(L^A)$ the first and last point in $L^A$ `visited' by $L$. 

\begin{defn}(Ballistic measure)
   Let $A\in\mP^d_f\setminus\{\emptyset\}$. For $z\in \Z^d$ and $L\in A(d,1)$ define
    \begin{align*}
      b_A(z,L)&:=\begin{cases}
        1, & \text{if } L^A=\{z\},\\
        \frac 12, & \text{if } \#(L^A)>1 \text{ and } z\in\{\min(L^A),\max(L^A)\},\\
        0, & \text{otherwise}.
      \end{cases}
    \end{align*}
   Then a probability measure $b_A$ on $A$ is well defined by setting
    \begin{align*} 
      b_A(z)&:=
      \int_{A(d,1)} b_A(z,L) \P^{\sq A}(dL), \qquad z\in \Z^d.
    \end{align*}
We call $b_A$ the \emph{ballistic measure} on $A$.
\end{defn}
 Observe that $b_A(z)=0$ for $z\notin A$ and $\sum_{z\in A}b_A(z)=1$. Hence, $b_A$ is indeed a probability measure which is concentrated on $A$, that is, $b_A\in\sD(A)$. Note also that, for any $z\in A$,
\begin{align} \label{eq:upperbd-b_A} 
      b_A(z)&=
      \int_{A(d,1)} \hspace{-2mm} b_A(z,L) \P^{\sq A}(dL)\leq \int_{A(d,1)} \hspace{-2mm}\1\{L\in[C_z]\} \P^{\sq A}(dL)=\P^{\sq A}([C_z]). 
    \end{align}
    This gives a first upper bound of the probability mass of $b_A$ at single vertices.
     {\em Ballistic aggregation (BA) on $\Z^d$} is now defined to be incremental aggregation with distribution family $(b_A)_{A\in\mP^d_f}$, where $b_A$ is the ballistic measure on $A$.

For the ballistic model in $\Z^d$, $d\geq 2$, we have the following main result concerning the radial growth of clusters, which parallels the one obtained by Kesten for DLA.

 \begin{thm} \label{thm:ba-main} 
     Let $d\in\N$, $d\geq 2$ and $(F_n)_{n\in\N}$ be ballistic aggregation in $\Z^d$. There exists a constant $c>0$ such that almost surely for $n$ sufficiently large
     \begin{align*}
       \rad(F_n)\leq 
          c\, n^{1/2}.
     \end{align*}
  \end{thm}

 Combining this bound with the trivial lower bound provided by Propositon~\ref{prop:dim-triv-bounds}, we conclude that in $\Z^2$ the fractal dimension $\delta_f$ of BA clusters exists and equals $2$. This confirms a long standing conjecture in the physics literature, cf.~\cite{BDA83,Meakin83,BallWitten84,Vicsek89}.
  \begin{cor} \label{cor:fdim-is-2}
  For the ballistic model in $\Z^2$, one has $\delta_f=2$ almost surely.
\end{cor}
\begin{proof}
   On the one hand, by Propositon~\ref{prop:dim-triv-bounds}, $\udelta_f\leq 2$ almost surely. On the other hand, Theorem~\ref{thm:ba-main} implies
  \begin{align*}
     \ldelta_f=\liminf_{n\to\infty}\frac{\log n}{\log(\rad(F_n))}\geq \lim_{n\to\infty}\frac{ \log n}{\log(c n^{1/2})}=2. \qquad \qedhere
  \end{align*}
\end{proof}

\begin{rem} We can even get a slightly stronger conclusion from the above theorem regarding the asymptotic size of BA clusters.
  The ballistic model in $\Z^2$ satisfies almost surely
  $$
  \liminf_{n\to\infty}\frac{\# F_n}{(\rad(F_n))^2}\geq \liminf_{n\to\infty}\frac{n}{c^2 n}=c^{-2}> 0.
  $$
This may be interpreted as saying that asymptotically BA clusters cover a positive portion of the plane. (More precisely, the quotient of the area covered by the cluster $\sq F_n$ (given by $\# F_n$) and the area of the smallest ball containing $\sq F_n$ ($\approx \pi\cdot \rad(F_n)^2$) is asymptotically bounded from below, as $n\to\infty$.)
\end{rem}

For $d\geq 3$, the conclusion regarding the fractal dimension of the BA model is not quite as strong as for $d=2$.

 \begin{cor} \label{cor:fdim_case_d}
  For the ballistic model in $\Z^d$, $d\geq 3$, one has $\ldelta_f\geq 2$ almost surely.
\end{cor}
As in the case $d=2$, this can be directly deduced from Theorem~\ref{thm:ba-main}. Unfortunately, this lower bound for the fractal dimension is far from the conjectured value $\delta_f=d$ in this case.

The proof of Theorem~\ref{thm:ba-main} will be based on Kesten's method (Theorem~\ref{thm:Kestens-method} above). In order to apply it, a suitable estimate for the local growth of the ballistic measures $b_A$ is required, which we state now.  


\begin{prop}
  \label{prop:ballist-meas-bound-d} \label{prop:ballist-meas-bound}
    For any $d\in\N$, $d\geq 2$, there exists a constant $C_d>0$ such that, for any $r\geq 1$, any connected set $A\in \mP^d_f$ with $0\in A$ and $\rad(A)\geq r$, and any $z\in A$,
\begin{align*}
  b_A(z)\leq C_d\, r^{-1}.
\end{align*}
For $d=2$, $C_2=2$ is a suitable constant.
\end{prop}

Before we provide a proof of Proposition~\ref{prop:ballist-meas-bound}, we first clarify that Theorem~\ref{thm:ba-main} easily follows from this statement.

\begin{proof}
      [Proof of Theorem~\ref{thm:ba-main}] By Proposition~\ref{prop:ballist-meas-bound}, the hypothesis in Kesten's method (Theorem~\ref{thm:Kestens-method}) is satisfied for the exponent $q=1$ for the ballistic model $(F_n)_{n\in\N}$ on $\Z^d$ for any $d\geq 2$. Hence from this theorem, the stated bounds on the radial growth of this model (with the exponent $1/(1+q)=1/2$) follow at once.
    \end{proof}

    Observe that any improvement on the exponent $q=1$ in Proposition~\ref{prop:ballist-meas-bound} (i.e., any larger $q$) would allow to deduce from Kesten's method a better bound for the radial growth. Unfortunately, the exponent $q=1$ in Proposition~\ref{prop:ballist-meas-bound} turns out to be optimal in any dimension $d\geq 2$, as the following remark clarifies.
   \begin{rem}
      The exponent $q=1$ in Proposition~\ref{prop:ballist-meas-bound} is optimal (largest possible) in any space dimension $d\geq 2$. Consider for any $d\in\Z$, $d\geq 2$ and any integer $r\geq 1$ the set
            $$
            A=A_r:=\left\{(s,0,\ldots,0)\in\Z^d: s\in\{0,...,r\}\right\}.
            $$
        Observe that $\sq A$ is a union of $r+1$ unit size boxes placed side by side in a row. Hence $\sq A$ is convex and $\rad(A)=r$. Moreover, noting that for any directed line $L$ hitting the cube $C_0$ at all, $C_0$ is always the first or the last cube of $\sq A$ visited by $L$, we obtain for the ballistic measure at $z=0$ (and similarly at $z=(r,0,\ldots,0)$)
      \begin{align*}
      b_A(z) &=\int_{A(d,1)} b_A(z,L)\P^{\sq A}(dL)
      \geq\frac 12 \int_{A(d,1)} \hspace{-2mm}\1\{L\in[C_z]\} \P^{\sq A}(dL)
     = \frac 12\P^{\sq A}([C_z])\\ 
     &= \frac 12 \frac{\mu_1([C_z])}{\mu_1([\sq A])}
    =\frac 12 \frac{V_{d-1}(C_z)}{V_{d-1}(A_r)}
    =\frac 12\cdot\frac d{1+(d-1)(r+1)}= \frac 12\cdot \frac{d}{(d-1)r+d}.
    \end{align*}
    For any $r\geq d$, the denominator in the last expression is bounded from above by $d\,r$ and therefore we obtain $b_A(z)\geq 1/2\cdot r^{-1}$. Hence a bound as in Proposition~\ref{prop:ballist-meas-bound} valid for the ballistic measures $b_A$ of all finite sets $A\in\P^d_f$ cannot be true for any $q>1$.
   \end{rem}

\begin{proof}[Proof of Proposition~\ref{prop:ballist-meas-bound}]
We first discuss the case $d=2$, for which our argument is much simpler than for the general case.
Fix $r\geq 1$. Let $A\in\mP^2_f$ be a connected set with $0\in A$ and $\rad(A)\geq r$. 
 Note that the connectedness implies $[\sq A]=[\conv(\sq A)]$.
     The set $\sq A$ contains the square $C_0$ and another unit square $C_y$ with $y\in\Z^2$ and $\|y\|\geq r$. Hence $\conv(\sq A)$ contains $\conv(C_0\cup C_y)$ and thus a rectangle $R_r$ with sidelengths $r$ and $1$. (Choose one side of $R_r$ parallel to the vector $y$.) 
    Using \eqref{eq:upperbd-b_A} and the properties of the measure $\mu_1$,  we conclude that,  for any $z\in A$,
    \begin{align*}
      b_A(z)
     &\leq \P^{\sq A}([C_z])
     = \frac{\mu_1([C_z])}{\mu_1([\conv(\sq A)])}\leq \frac{\mu_1([C_0])}{\mu_1([R_r])}
     =\frac{V_1(C_0)}{V_1(R_r)}
     =\frac 2{1+r}\leq 2 r^{-1}.
    \end{align*}
    This shows the claimed estimate for $d=2$ and also that $C_2=2$ is a suitable constant in this case.

Now let us turn to the case $d\geq 3$.  We start similary as for $d=2$. Let $r\geq 1$ and let $A\in\mP^d_f$ be a connected set with $0\in A$ and $\rad(A)\geq r$.
     As before, the set $\sq A$ contains the square $C_0$ and another unit square $C_y$ with $y\in\Z^d$ and $\|y\|\geq r$, but now we can not simply replace $\sq A$ by its convex hull, since the latter set may be hit by significantly many more lines than $\sq A$.

      Since $A$ is connected, it must contain a path from $0$ to $y$, that is, there are $m\in\N$ and points $p_0:=0,p_1,\ldots,p_{m-1},p_m:=y\in A$ such that $||p_i-p_{i-1}||=1$ for $i=1,\ldots,m$. Let $\Gamma:=\{p_0,\ldots,p_m\}$ and let $\graphG$ denote (the graph of) the shortest curve connecting these points in the given order. Notice that $\graphG$ consists of axis-parallel segments of length 1. For $F\subset\R^d$ and $r>0$ let
      $$
      F_{\oplus r}:=\left\{x\in\R^d: \inf_{y\in F}||x-y||\leq r\right\}
      $$
      denote the $r$-parallel set of $F$.
      Observe that $\graphG_{\oplus 1/2}\subset \sq \Gamma$. (Indeed, any point in $\graphG_{\oplus 1/2}$ is contained in some ball $B(x,\frac 12)$ with $x\in\graphG$. If $x\in\Gamma$, then $B(x,\frac 12)\subset C_x$ and if $x$ is on the segment $S(p_k,p_{k+1})$ connecting $p_k$ and $p_{k+1}$, then obviously $B(x,\frac 12)\subset C_{p_k}\cup C_{p_{k+1}}$.) This yields
      $$
      \mu_1([\graphG_{\oplus 1/2}])\leq \mu_1([\sq \Gamma])\leq \mu_1([\sq A]).
      $$
      We claim now that there is a constant $\widetilde C_d$ (independent of $G$ and $r$) such that
      \begin{align}
        \label{eq:lower-bd-claim}
        \mu_1([\graphG_{\oplus 1/2}])\geq \widetilde C_d\, r.
      \end{align}

    Using \eqref{eq:upperbd-b_A}, the properties of the measure $\mu_1$ and \eqref{eq:crofton}, we conclude from \eqref{eq:lower-bd-claim} that,  for any $z\in A$,
    \begin{align*}
      b_A(z)
     &\leq \P^{\sq A}([C_z])
     = \frac{\mu_1([C_z])}{\mu_1([\sq A])}\leq \frac{\mu_1([C_0])}{\mu_1([\graphG_{\oplus 1/2}])}
     \leq\frac{\alpha_d V_{d-1}(C_0)}{\widetilde C_d\, r}
     =\frac {\alpha_d\, d}{\widetilde C_d} r^{-1}.
    \end{align*}
    This shows the estimate stated in Proposition~\ref{prop:ballist-meas-bound} (for the constant $C_d:=\alpha_d d/\widetilde C_d$).

    It remains to provide a proof of \eqref{eq:lower-bd-claim}. The rough idea here is to compare the measure of lines through the tube $G_{\oplus 1/2}$ with the measure of lines through the tube $S_{\oplus 1/2}$, where $S:=S(0,y)$ is the segment connecting $0$ and $y$, and to observe that the latter is of order $r$. Observe that any hyperplane hitting $S$ does also hit the curve $\graphG$, since both curves have the same endpoints. (Indeed, if the hyperplane contains an endpoint of $S$ then it obviously intersects also $G$. Otherwise the hyperplane contains an inner point of $S$ and separates the common endpoints of $S$ and $G$. But then the curve $G$ must also intersect the hyperplane when passing from one side to the other.) This will be useful for decomposing $\mu_1([\graphG_{\oplus 1/2}])$.

    Let $A(d,d-1)$ denote the set of all $(d-1)$-flats and $\mu_{d-1}$ the unique Euclidean motion invariant measure on $A(d,d-1)$ such that $\mu_{d-1}(\left[B^d\right])=2$, see \cite[Ch.~13.2]{SchneiderWeil} for details. For $H\in A(d,d-1)$ denote by $A(H,1)$ the set of lines in $H$ and let $\mu_1^H$ be the invariant line measure on $A(H,1)$ (which is the image measure of $\mu_1$ on $\R^{d-1}$ under some isometry from $\R^{d-1}$ to $H$. Recall that for $d\geq 3$ the measure of any set $\sA$ of lines may be determined by computing its line measure in hyperplanes and then integrating over all hyperplanes. More precisely, for any $\sA\in\sA(d,1)$,
    \begin{align*}
      \mu_1(\sA)=\int_{A(d,d-1)} \mu_1^H(\sA) \mu_{d-1}(dH),
    \end{align*}
    see e.g.\ \cite[Thm.~7.1.2 and the subsequent Remark]{SchneiderWeil}.

    Let $s:=y/||y||$ be the direction of the segment $S$. For any $H\in A(d,d-1)$, denote by $v_H\in\bS^{d-1}$ some unit normal of $H$ and by $p_H$ the corresponding (signed) distance to $0$ (such that $H=\{x\in\R^d: \langle x,v_H\rangle=p_H\}$). Note that $v_H$ is uniquely determined for a.a.\ $H$ if we require additionally  $\langle s,v_H\rangle\geq 0$. For a lower bound, we do not need to compute the line measure in all hyperplanes. We restrict our attention to the following set. Fix some $\delta>0$ and let
    \begin{align*}
      \sH:=\left\{H\in A(d,d-1): \langle v_H, s\rangle \geq \delta \text{ and } H\cap S\neq\emptyset \right\}.
    \end{align*}
    Observe that, for each $H\in\sH$, we have $\graphG\cap H\neq\emptyset$, i.e.\ there is some point $x\in G\cap H$, implying that $\graphG_{\oplus 1/2}$ contains the closed ball $B(x,\frac 12)$ with radius $\frac 12$ and center $x$. Hence, by \eqref{eq:crofton},
    \begin{align*}
       \mu_1^H\left(\left[\graphG_{\oplus 1/2}\right]\right)\geq \mu_1^H\left(\left[B(x,1/2)\right]\right)=\alpha_{d-1} V_{d-2}(B^{d-1}(1/2))=:\widetilde c_d,
    \end{align*}
    where $B^{k}(t)$ denotes a ball in $\R^k$ with radius $t$ (centered at $0$). The most important thing to note here is that $\widetilde c_d$ is some positive constant independent of $H$ and $r$. Therefore, we infer
    \begin{align} \label{eq:Gtube}
       \mu_1\left(\left[\graphG_{\oplus 1/2}\right]\right)
       &= \int_{A(d,d-1)} \mu_1^H\left(\left[\graphG_{\oplus 1/2}\right]\right) \mu_{d-1}(dH)\\ \notag
       &\geq\int_{\sH} \mu_1^H\left(\left[\graphG_{\oplus 1/2}\right]\right) \mu_{d-1}(dH)
       \geq 
       \widetilde c_d\, \mu_{d-1}(\sH).
    \end{align}
    Employing \cite[Theorem 13.2.12]{SchneiderWeil}, we conclude that
    \begin{align*}
      \mu_{d-1}(\sH)&=\int_{G(d,d-1)}\int_{F^{\perp}} \ind{\sH}(F+z) \lambda_1(dz)\nu_{d-1}(dF)\\
      &=\int_{G(d,d-1)}\ind{}\{\langle s,v_F\rangle\geq\delta\}\int_{F^{\perp}} \ind{}\{(F+z)\cap S\neq\emptyset\} \lambda_1(dz)\nu_{d-1}(dF),
    \end{align*}
    where $G(d,d-1)$ denotes the Grassmannian of hyperplanes in $\R^d$ and $\nu_{d-1}$ the invariant measure on $G(d,d-1)$. 
     Moreover, as before $v_F$ denotes the unit normal of $F$ such that $\langle s,v_F\rangle\geq 0$ and $\lambda_1$ is the $1$-dim.\ Lebesgue measure on $F^\perp$. Now observe that the inner integral in the last term is bounded from below by $r\delta$. Indeed, we have $(F+z)\cap S\neq\emptyset$ if and only if $\langle y, v_F\rangle\geq \langle z,v_F\rangle\geq 0$. Moreover, if the first indicator is 1, then $\langle y,v_F\rangle= ||y|| \langle s,v_F\rangle\geq r\delta$. Hence $\lambda_1(\{z\in F^\perp:(F+z)\cap S\neq\emptyset\})\geq \lambda_1(\{z\in F^\perp:r\delta\geq\langle z,v_F\rangle\geq 0\})=r\delta$ as claimed.

    It follows that
    \begin{align} \label{eq:sH2}
      \mu_{d-1}(\sH)&\geq r\delta \cdot \nu_{d-1}(\{F\in G(d,d-1): \langle v_F,s\rangle\geq\delta\})=:\hat c_d\, r,
    \end{align}
    where the constant $\hat c_d$ is positive and independent of $r$. Now the claim \eqref{eq:lower-bd-claim} follows (for the constant $\widetilde C_d:=\widetilde c_d\,\hat c_d$) by combining \eqref{eq:Gtube} and \eqref{eq:sH2}.
    \end{proof}

\section{Proof of Kesten's method} \label{sec:proofs}

The first step towards a proof of Theorem~\ref{thm:Kestens-method} is an equivalent reformulation of the conclusion in a more convenient form.
Let $(F_n)_{n\in\N}$ be some incremental aggregation.
Recall the definition of the radius $\rad(F_n)$ from \eqref{eq:def-rad}.
We define the two random functions
\begin{align*}
	R: \N \to [0,\infty),\quad n\mapsto \rad(F_n)
\end{align*}
and
\begin{align*}
	T: [0,\infty) \to \N,\quad r\mapsto \min\{j\in\N\ |\ R(j)\geq r\}.
\end{align*}
Clearly, $T(r)$ is the first time step at which the growing cluster has radius at least $r$.
	It is easy to see that almost surely the functions $R$ and $T$ are non-decreasing. 
Moreover, one has almost surely the relations $T(R(n)) \leq n$ for each $n\in\N$, and $R(T(r)) \geq r$ for each $r\in [0,\infty)$.
Note also that almost surely
	\begin{align*}
		R(n) \to\infty \text{ as } n\to\infty
	\quad \text{ and }\quad
		T(r) \to\infty \text{ as } r\to\infty.
	\end{align*}
The conclusion in Theorem~\ref{thm:Kestens-method} is an almost sure upper bound on the random function $R$. Our aim is to reformulate this upper bound on $R$ equivalently as a lower bound on the waiting time $T$. The following statement provides this link. Recall that a function $h:[0,\infty) \to [0,\infty)$ is called multiplicative if and only if $h(xy)=h(x)h(y)$ for all $x,y\geq 0$. Note that multiplicativity implies $h(0)=0$ and $h(1)=1$. (The functions $x\mapsto x^p$, $p>0$ are generic examples for such $h$.)

\begin{lem} \label{randt}
	Let $a>0$ and $h:[0,\infty) \to [0,\infty)$ be a bijective, multiplicative and increasing function. For $c>0$ define the events
	\begin{align*}
		A_c:=\{\omega \in\Omega\ |\ \exists N\in\N: R(\omega)(n) \leq c\,h(n) \text{ for all } n\geq N\}
	\end{align*}
	and
	\begin{align*}
		D_c := \{\omega \in\Omega\ |\ \exists r_0\in\N: T(\omega)(ar)\geq c\,h^{-1}(r) \text{ for all } r\geq r_0\}.
	\end{align*}
	Then the following assertions are equivalent:
	\begin{enumerate}
		\item[(i)] There exists a constant $c>0$ such that $\P(A_c) = 1$.
		\item[(ii)] There exists a constant $\bar c>0$ such that $\P(D_{\bar c})=1$.
	\end{enumerate}
\end{lem}

\begin{proof}
	Fix some $c>0$ and let $\omega\in A_c$. Choose $N\in\N$ such that $R(n)\leq c\,h(n)$ holds for all $n\geq N$. (Here and in the sequel $N$, $R$ and $T$ depend on $\omega$, which we suppress in the notation.) Since $h$ is bijective and multiplicative, the last inequalities are equivalent to $\tilde c\,h^{-1}(R(n))\leq n$ being satisfied for all $n\geq N$, where $\tilde c:=h^{-1}(1/c)$. Since $T(s)\to\infty$ as $s\to\infty$, we can find $r_0\in\N$ large enough such that $T(ar_0) > N$. Obviously this implies $T(ar)\geq T(ar_0) > N$ for all $r\geq r_0$, since $T$ is increasing. Hence we infer that $\tilde c h^{-1}(R(T(ar))) \leq T(ar)$ holds for each $r\geq r_0$. Observe now that $h^{-1}$ is increasing and multiplicative since $h$ is. Taking into account that $R(T(ar))\geq ar$, we conclude that for each $r\geq r_0$
$$
T(ar)\geq \tilde c h^{-1}(ar)=\tilde c h^{-1}(a)\cdot h^{-1}(r)=:\bar c\cdot h^{-1}(r).
$$
Hence $\omega \in D_{\bar c}$. It is now important to note that $\bar c=h^{-1}(a/c)$ does not depend on $\omega$ (but only on $h$, $a$ and $c$). Therefore we have proved that $A_c\subseteq D_{\bar c}$.

With a completely analogous argument one can show that $D_{\bar c}\subseteq A_c$. Hence $A_c=D_{\bar c}$ for any $c>0$ (and $\bar c=h^{-1}(a/c)$). In particular, if $\P(A_c)=1$ for some $c>0$, then $\P(D_{\bar c}) = 1$. Similarly, if $\P(D_{\hat c}) = 1$ for some $\hat c>0$, then $c:=a/h(\hat c)>0$, $\bar c=h^{-1}(a/c)=\hat c$ and therefore $\P(A_c=1)=1$. This completes the proof.
\end{proof}

Assume now that the hypothesis of Theorem~\ref{thm:Kestens-method} is satisfied for some $q>0$.
Consider for any $c>0$ the events
	\begin{align*}
		A_c := \{\omega\in\Omega\ |\ \exists N\in\N: R(\omega)(n) \leq c\, n^{\frac{1}{1+q}} \text{ for all } n\geq N\}
	\end{align*}
	and
	\begin{align*}
		D_c := \{\omega\in\Omega\ |\ \exists r_0\in\N: T(\omega)(2r) \geq c r^{1+q} \text{ for all } r\geq r_0 \}.
	\end{align*}
It is clear that the conclusion of Theorem~\ref{thm:Kestens-method} may be formulated as follows in terms of the events $A_c$:
\emph{There exists a constant $c>0$ such that $\P(A_c)=1$.}

Note that $A_c$ and $D_c$ are the same as the corresponding events defined in Lemma~\ref{randt}	for the choices $a=2$ and \begin{align*}
		h: [0,\infty) \to [0,\infty), \quad x\mapsto x^{\frac{1}{1+q}}.
	\end{align*}
Clearly, $h$ is bijective, multiplicative and increasing as required. Hence, by Lemma~\ref{randt}, the conclusion of Theorem~\ref{thm:Kestens-method} holds if and only if
there exists a constant $\beta>0$ such that
	\begin{flalign} \label{lambda2}
		\P(D_\beta) = 1.
	\end{flalign}
Therefore, in order to prove Theorem~\ref{thm:Kestens-method} it clearly suffices to prove that its hypothesis implies the existence of some constant $\beta>0$ such that \eqref{lambda2} holds.

In order to achieve this we introduce some more notation. For $n\in\N$ we write $F_n = \{y_1,\dots,y_n\}$, where $y_j$ denotes the (random) point added to the cluster at time $j\in\N$. Fix some $\beta > 0$, which will be determined later on. For $r\in\N$ let $\widetilde{m}_r := \beta r^{1+q}$ and define
	\begin{align*}
		V_r := \{\omega\in\Omega\ |\ T(\omega)(2r) < \widetilde{m}_r\}.
	\end{align*}
	Let $S_r := \{x\in\Z^d\ |\ r \leq |x| < r+1\}$ be the \emph{discrete sphere} in $\Z^d$ of radius $r$ centered at the origin.  (It is easy to see that any connected set $A\in\Z^d$ with $0\in A$ and $\rad(A)\geq r$ will contain at least one point of $S_r$, since there will be a path in $A$ connecting $0$ to some point at distance at least $r$ which can not avoid intersecting the set $S_r$.) Further we define the set of all possible (self-avoiding) paths in $\Z^d$ of length $r$ (i.e., passing $r$ edges) with starting point in $S_r$ by 
	\begin{align} \label{eq:def-Zr}
		Z_r := \big\{\bz := (z_0, \dots, z_r)\in (\Z^d)^r \ |\ &z_1\in S_r, \{z_i,z_{i+1}\}\in E \text{ for } i=0,\ldots,r-1 \notag\\ & \text{ and } z_i\neq z_j \text{ for }i\neq j\big\}.
	\end{align}
	For any $\bz\in Z_r$ we define the event
	\begin{flalign*}
		W_r(\bz) := \left\{\omega\in\Omega\ |\ \exists j_0< \dots < j_r \leq \widetilde{m}_r \ \text{ such that } y_{j_i}(\omega)  = z_i \text{ for } i=0,\ldots,r \right\}
	\end{flalign*}
	and the union of these events by
	\begin{align*}
		W_r := \bigcup_{\bz \in Z_r} W_r(\bz).
	\end{align*}
	Note that this is a finite union, as for any $r\in\N$ there are only finitely many paths in $Z_r$. Note also that $W_r$ as well as $V_r$ depend (via $\widetilde{m}_r$) on the choice of $\beta$, which is suppressed in the notation but which will become important later. We claim that for each $r\in\N$
	\begin{align}\label{eq:VsubsetW}
		V_r \subseteq W_r.
	\end{align}
	Indeed, if  $\omega \in V_r$, then $T(\omega)(2r) < \widetilde{m}_r$. For $m_r:=\max\{j\in\N\ |\ j\leq \widetilde{m}_r\}$ this implies $T(\omega)(2r)\leq m_r$, and hence $\rad(F_{m_r}(\omega)) \geq 2r$. Thus $F_{m_r}(\omega)$ contains a point $x$ with $\|x\|\geq 2r$. Since $F_{m_r}(\omega)$ is connected and contains the origin, there exists a path $\mathbf{p}=(y_{i_1},\ldots,y_{i_\ell})$ in $F_{m_r}(\omega)$ connecting $0$ and $x$ (i.e., $y_{i_0}=0$ and $y_{i_\ell}=x$) with the additional property that $i_0<i_1\ldots<i_\ell$, that is, the points of $\mathbf{p}$ have been added to the cluster in the right order. (Not every path between $0$ and $x$ might satisfy this, but such a path must exist, since otherwise the cluster would be disconnected at some time.) $\mathbf{p}$ contains a point $z_0$ of $S_r$ and thus a subpath $\mathbf{p}'$ connecting $z_0$ to $x$. Since $\|z_1-x\|>r-1$, $\mathbf{p}'$ has length at least $r$. Hence, by taking the first $r$ steps of $\mathbf{p}'$, we find some path $\bz$ of length exactly $r$, that is, some $\bz\in Z_r$ in $F_{m_r}(\omega)$. This means $\omega \in W_n(\bz) \subset W_n$ which proves the set inclusion \eqref{eq:VsubsetW}. 

The following estimate will allow to complete the proof of \eqref{lambda2}.

\begin{lem} \label{lem:Wr-bound}
  There is some $\beta>0$ and constants $c_1>0$, $q\in(0,1)$ and $r_0\in\N$  (which depend on $\beta$) such that for all $r\in\N$ with $r\geq r_0$
  \begin{align*}
    \P(W_r)\leq c_1 r^{d-1} q^r.
  \end{align*}
\end{lem}

Before we prove Lemma~\ref{lem:Wr-bound} let us discuss how it is used to complete the proof of \eqref{lambda2}.
Let us fix $\beta>0$ as required for applying the lemma. Then, clearly, since $p<1$, we have
 \begin{align*}
    \sum_{r=1}^\infty \P(W_r)\leq r_0+\sum_{r\geq r_0} \P(W_r)\leq r_0+\sum_{r\geq r_0} c_1 r^{d-1} p^r<\infty
  \end{align*}
and thus, by \eqref{eq:VsubsetW},
\begin{align*}
    \sum_{r=1}^\infty \P(V_r)<\infty.
  \end{align*}
  Hence, by the Borel-Cantelli lemma,
		$\P(\limsup_{r\to\infty} V_r) = 0.$
	Since
	\begin{align*}
		(\limsup_{r\to\infty} V_r)^c &= \{\omega\in\Omega\ |\ \exists r_0\in\N \text{ s.t. } \omega\in V_r^c \text{ for all }r>r_0\}\\
&= \{\omega\in\Omega\ |\ \exists r_0\in\N \text{ s.t. } T(\omega)(2r) \geq \beta r^{1+q} \text{ for all }r\geq r_0\}= D_\beta
	\end{align*}
	we  conclude that $\P(D_\beta) = 1$. This completes the proof of \eqref{lambda2} and thus of Kesten's method (up to a proof of Lemma~\ref{lem:Wr-bound}).

   In order to prepare the proof of Lemma~\ref{lem:Wr-bound}, we introduce some more notation and provide a few auxiliary results.
For any $r\in \N$, any path $\bz=(z_0,\ldots,z_r)\in Z_r$ (cf.~\eqref{eq:def-Zr}) and any $i\in \{0,\dots,r\}$, we define random variables
	\begin{flalign*}
		\tau^{\bz}_i:\Omega \to \N_\infty, \quad \tau^{\bz}_i(\omega) = \begin{cases} j, &
			\text{ if } j\in\N \text{ and } y_j(\omega) = z_i,\\
			\infty, &\text{ if } z_i \notin F_\infty(\omega),
		\end{cases}
	\end{flalign*}
	that is, $\tau^{\bz}_i$ is either the time point $j$ at which $z_i$ is added to the cluster or $\infty$ if this never happens.
Further, for $i\in \{1,\dots,r\}$, we define waiting times
	\begin{flalign*}
		\sigma^{\bz}_i: \Omega \to \N_\infty, \quad \omega\mapsto \begin{cases}
			\tau^{\bz}_{i}(\omega) - \tau^{\bz}_{i-1}(\omega), &\text{ if } \tau^{\bz}_{i-1}(\omega)<\tau^{\bz}_{i}(\omega) < \infty, \\
			\infty, &\text{ else},	
		\end{cases}
	\end{flalign*}
	so $\sigma^{\bz}_i$ is the waiting time between adding $z_{i-1}$ and $z_{i}$ to the cluster, if both are added and $z_{i-1}$ is added before $z_{i}$. 
We also consider the event
	\begin{flalign*}
		U_{\bz} := \{\tau^{\bz}_0 < \ldots < \tau^{\bz}_r<\infty\}
	\end{flalign*}
that all the vertices of the path $\bz$ 
are added to the cluster in the right order (i.e., $z_\ell$ before $z_{\ell+1}$). Note that $\P(U_{\bz})>0$.  We will now prove that the random variable $\sigma^{\bz}_i$ conditioned on $U_{\bz}$ always dominates a geometrically distributed random variable with parameter
	\begin{flalign} \label{eq:geom2}
		p_r := C r^{-q},
	\end{flalign}
	where the constants $C$  and $q$ are the ones from the hypothesis in Theorem\ref{thm:Kestens-method}. In particular, they are the same for all $i\in \{0,\dots,r-1\}$ and independent of $r\in\N$.
\begin{lem}\label{lem:geom-dom}
   For any $r\in\N$ and $i=1,\ldots,r$, we have
   \begin{align*}
     \P(\sigma^{\bz}_i> k|U_{\bz})\geq\P(Y> k)
   \end{align*}
   for any $k\in\N$, where $Y$ is geometrically distributed with parameter $p_r$ as in \eqref{eq:geom2}.
\end{lem}
\begin{proof}
  Fix $i\in\{1,\ldots,r\}$. Let $U_{\bz}^{i,j}:=U_{\bz}\cap\{\tau^{\bz}_i=j\}$, $j\in\N$ and let $k\in\N$. It suffices to show that for any $j$ such that $\P(U_{\bz}^{i,j})>0$
  \begin{align}\label{eq:geom-dom-proof1}
     \P(\sigma^{\bz}_i> k|U_{\bz}^{i,j})\geq\P(Y> k).
   \end{align}
  Indeed, since $U_{\bz}=\bigcup_{j\in\N} U_{\bz}^{i,j}$ is a disjoint decomposition, we infer that
  \begin{align*}
    \P(\sigma^{\bz}_i > k| U_{\bz})
    &=\sum_{j\in\N} \P(\{\sigma^{\bz}_i> k\}\cap\{\tau^{\bz}_i=j\}| U_{\bz})\\
    &=\sum_{j\in\N} \frac{\P(\{\sigma^{\bz}_i> k\}\cap U_{\bz}^{i,j})}{\P(U_{\bz})}
    =\sum_{j:\P(U_{\bz}^{i,j})>0} \frac{\P(\{\sigma^{\bz}_i> k\}\cap U_{\bz}^{i,j})}{\P(U_{\bz})}\\
    &=\sum_{j:\P(U_{\bz}^{i,j})>0} \P(\sigma^{\bz}_i> k|U_{\bz}^{i,j})\frac{\P(U_{\bz}^{i,j})}{\P(U_{\bz})}.
  \end{align*}
Now, by \eqref{eq:geom-dom-proof1}, the first probability in each summand can be bounded from below by $\P(Y>k)$ and taking this factor out of the summation, the remaining summands sum to 1. Hence, we obtain
   $ \P(\sigma^{\bz}_i> k| U_{\bz})\geq \P(Y>k)$
  as asserted in the lemma.

  For a proof of \eqref{eq:geom-dom-proof1}, 
  fix $j\in\N$ and
  define for any $\ell\in\N$ the event $A_\ell:=\{\omega\in\Omega: y_{j+\ell}(\omega)\neq z_{i}\}$. Set $A_0:=U_{\bz}^{i,j}$. Then, by the multiplication formula of conditional probabilities, we have for any $k\in\N$
  \begin{align*}
    \P(\sigma^{\bz}_i>k|U_{\bz}^{i,j})
    &=\P(\bigcap_{\ell=1}^k A_\ell| A_0)
    =\prod_{m=1}^k \P(A_m|\bigcap_{\ell=0}^{m-1} A_\ell).
  \end{align*}
  Now observe that the condition in the $m$-th factor implies, that at time $j+m$ the cluster $F_{j+m}$ has radius at least $r$ and contains $z_{i-1}$ but not $z_{i}$. (And thus, due to the choice of $\bz$, $z_{i}$ is in the outer boundary of $F_{j+m}$.) Hence, by the hypothesis \eqref{eq:q-estimate} in Kesten's method (Theorem~\ref{thm:Kestens-method}), we get
  \begin{align*}
    \P(A_m|\bigcap_{\ell=0}^{m-1} A_\ell)\geq 1- p_r.
  \end{align*}
  We conclude that $\P(\sigma^{\bz}_i>k|U_{\bz}^{i,j})\geq (1-p_r)^k=\P(Y>k)$ holds for any $k\in\N$, proving \eqref{eq:geom-dom-proof1}.
\end{proof}

Since we want to sum geometric random variables in a moment, we also provide the following standard bound, which can e.g.\ be found in \cite[Lemma~2.6.2]{lawler91}.
\begin{lem} \label{lem:geometric}
	Let $n\in\N$, $n\geq 2$ and let $Y_1,\dots,Y_{n}$ be independent geometrically distributed random variables with parameter $0<p<\frac{1}{2}$. 
Then for every $a\geq 2p$,
	\begin{align*}
		\P\left(\sum_{i=1}^{n}Y_i\leq a\frac{n}{p}\right) \leq (2ae^2)^n.
	\end{align*}
\end{lem}

Based on this important observation we obtain the following bound for the waiting time $\tau^{\bz}_r$ until the path $\bz\in Z_r$ is completely added to the cluster conditioned on the event that it is added.
%
%
\begin{lem} \label{lem:Wrz-bound}
  For any $\beta>0$ there is some $r_0\in\N$
  such that for all $r\in\N$ with $r\geq r_0$
  and all $\bz\in Z_r$
  \begin{align*}
    \P(W_r(\bz))\leq  (2\beta C e^2)^r,
   \end{align*}
 where $C$ is the constant in the hypothesis of Theorem~\ref{thm:Kestens-method}.
\end{lem}
\begin{proof}
      Fix some $\beta>0$ and set $a:=\beta C$. Choose $r_0\in\N$ large enough such that $p_{r_0}<\min\{\frac 12, \frac a2\}$.  Note that this implies $a=\beta C \geq 2 p_r$ for all $r\geq r_0$.

       Observe that $W_r(\bz)\subset\{\tau^{\bz}_r\leq \tilde m_r\}\cap U_{\bz}$ and thus
      \begin{align*}
         \P(W_r(\bz))\leq \P(W_r(\bz)|U_{\bz})\leq \P(\tau^\bz_r\leq \tilde m_r|U_{\bz})
      \end{align*}
      Note that under the condition $U_{\bz}$, $\tau^{\bz}_r$ may be represented as
    $\tau^{\bz}_r=\tau^{\bz}_0+\sum_{i=1}^{r} \sigma^{\bz}_i$. Hence, for any $t>0$, we obtain
    \begin{align*}
      \P(\tau^{\bz}_r\leq t|U_{\bz})&\leq\P(\tau^{\bz}_r-\tau^{\bz}_0\leq t|U_{\bz})=\P(\sum_{i=1}^{r} \sigma^{\bz}_i \leq t|U_{\bz}).
    \end{align*}
    Now recall from Lemma~\ref{lem:geom-dom} that, conditioned on $U_{\bz}$, each of the random variables $\sigma^{\bz}_i$ dominates a geometric random variable $Y$ with parameter $p_r$. Moreover, given $\tau_i^{\bz}$, the randomness in the variable $\sigma^{\bz}_{i+1}$ will only depend on the subsequent steps of the construction. Hence their sum dominates a sum $\sum_{i-1}^{r} Y_i$ of independent geometric random variables $Y_i$ with parameter $p_r$. That is, for any $t>0$
    \begin{align*}
       \P(\tau^{\bz}_r\leq t|U_{\bz})&\leq\P\left(\sum_{i=1}^{r} Y_i\leq t\right).
    \end{align*}
    Specializing now to $t=\tilde m_r$ and recalling from the definition of $\tilde m_r$ and \eqref{eq:geom2} that $\tilde m_r=\beta r^{1+q}=\beta C \frac r {p_r}=a\frac r{p_r}$, we can apply Lemma~\ref{lem:geometric} with $n=r$, $p=p_r$ and $a=\beta C$. (Note that the hypothesis $a\geq 2p_r$ is satisfied for $r\geq r_0$.)
    We obtain that for any $r\in\N$ with $r\geq r_0$,
     \begin{align*}
       \P(W_r(\bz))&\leq \P(\tau^{\bz}_r\leq \tilde m_r|U_{\bz})\leq \P(\tau^{\bz}_r\leq \frac{ar}{p_r}|U_{\bz})\leq (2\beta C e^2)^r,
    \end{align*}
    which completes the proof.
   \end{proof}

   Now we have all the ingredients to provide a proof of Lemma~\ref{lem:Wr-bound}.
   \begin{proof}[Proof of Lemma~\ref{lem:Wr-bound}]
      Choose $\beta>0$ small enough that $4d\beta C e^2<1$. By definition of $W_r$ and Lemma~\ref{lem:Wrz-bound}, there exists $r_0\in\N$ such that, for any $r\geq r_0$,
      \begin{align*}
		\P(W_r)&\leq \sum_{\bz \in Z_r} \P(W_r(\bz))\leq \# Z_r \cdot (2\beta C e^2)^r.
	\end{align*}
Recall that the paths of $Z_r$ start in $S_r$ and observe that there is some constant $c_1$ such that $\# S_r\leq c_1 r^{d-1}$. 
Hence there are at most $c_1 r^{d-1}$ starting points for paths in $Z_r$. Moreover, since each point in $\Z^d$ has $2d$ neighbors and paths in $Z_r$ have length $r$, there are at most $(2d)^{r}$ possible paths with a given starting point. Hence $\# Z_r\leq c_1 r^{d-1}(2d)^{r}$. We infer that for any $r\geq r_0$
\begin{align*}
		\P(W_r)&\leq  c_1 r^{d-1}(2d)^{r}\cdot  (2\beta C e^2)^r=c_1 r^{d-1} q^r,
	\end{align*}
where $q:=4d\beta C e^2<1$, due to the choice of $\beta$, and $c_1>0$. Hence we have found some $\beta>0$ (and suitable constants $c_1$, $q$ and $r_0$) such that the asserted bound in Lemma~\ref{lem:Wr-bound} holds for all $r\in\N$ with $r\geq r_0$.
   \end{proof}

 \bibliographystyle{plainurl}
\bibliography{DLA1}

\end{document}